\documentclass[12pt,leqno]{article}
\usepackage{empheq}
\usepackage{tikz}
\usetikzlibrary{matrix,arrows,decorations.pathmorphing}
\usepackage[doc]{optional}
\usepackage{color}
\usepackage{float}
\usepackage{soul}
\usepackage{url}
\usepackage{graphicx}
\definecolor{labelkey}{rgb}{0,0.08,0.45}
\definecolor{refkey}{rgb}{0,0.6,0.0}
\definecolor{Brown}{rgb}{0.45,0.0,0.05}
\definecolor{lime}{rgb}{0.00,0.8,0.0}
\definecolor{lblue}{rgb}{0.8,0.85,1.00}

\definecolor{myblue}{rgb}{.8, .8, 1}
  \newcommand*\mybluebox[1]{%
    \colorbox{myblue}{\hspace{1em}#1\hspace{1em}}}

\usepackage{amsmath}
\usepackage{stmaryrd}
\usepackage{amssymb}
\usepackage{theorem}

\usepackage[]{hyperref}

\newcommand{\nnn}{\ensuremath{{n\in{\mathbb N}}}}
\newcommand{\thalb}{\ensuremath{\tfrac{1}{2}}}
\newcommand{\menge}[2]{\big\{{#1}~\big |~{#2}\big\}}

\newcommand{\To}{\ensuremath{\rightrightarrows}}

\newcommand{\fenv}[1]%
{\ensuremath{\,\overrightarrow{\operatorname{env}}_{#1}}}
\newcommand{\benv}[1]%
{\ensuremath{\,\overleftarrow{\operatorname{env}}_{#1}}}

\newcommand{\scal}[2]{\left\langle{#1},{#2}  \right\rangle}

\newcommand{\RR}{\ensuremath{\mathbb R}}
\newcommand{\RP}{\ensuremath{\mathbb{R}_+}}
\newcommand{\RM}{\ensuremath{\mathbb{R}_-}}

\newcommand{\dom}{\ensuremath{\operatorname{dom}}}

\newcommand{\gr}{\ensuremath{\operatorname{gra}}}

\newcommand{\ran}{\ensuremath{\operatorname{ran}}}

\newcommand{\Fix}{\ensuremath{\operatorname{Fix}}}
\newcommand{\Id}{\ensuremath{\operatorname{Id}}}

{\begin{list}{}{%
\settowidth{\labelwidth}{\textrm{#1~}}%
\setlength{\leftmargin}{\labelwidth+\labelsep}}}
{\end{list}}
\newtheorem{theorem}{Theorem}[section]
\newtheorem{lemma}[theorem]{Lemma}

\newtheorem{corollary}[theorem]{Corollary}

\newtheorem{proposition}[theorem]{Proposition}

\newtheorem{definition}[theorem]{Definition}

\theoremstyle{plain}{\theorembodyfont{\rmfamily}
}
\theoremstyle{plain}{\theorembodyfont{\rmfamily}
}
\theoremstyle{plain}{\theorembodyfont{\rmfamily}
}
\theoremstyle{plain}{\theorembodyfont{\rmfamily}
\newtheorem{example}[theorem]{Example}}
\newtheorem{ex}[theorem]{Example}
\newtheorem{fact}[theorem]{Fact}
\theoremstyle{plain}{\theorembodyfont{\rmfamily}
\newtheorem{remark}[theorem]{Remark}}

\renewcommand{\iff}{\Leftrightarrow}

\providecommand{\RR}{\mathbb{R}}

\providecommand{\ran}{\operatorname{ran~}}

\providecommand{\dom}{\operatorname{dom}~}

\newcommand{\fix}{\ensuremath{\operatorname{Fix}}}

\providecommand{\gr}{\operatorname{gr}}

\providecommand{\Id}{\operatorname{{ Id}}}

\providecommand{\bK}{{\bf K}}
\providecommand{\bZ}{{\bf Z}}

\providecommand{\To}{\rightrightarrows}

\providecommand{\gr}{\operatorname{gra}}
\providecommand{\fix}{\operatorname{Fix}}
\providecommand{\ran}{\operatorname{ran}}

\providecommand{\Id}{\operatorname{Id}}

\providecommand{\inns}[2][w]{ [#2; #1]}
\newcommand{\outs}[2][w]{[#1; #2]}
\providecommand{\J}{{ J}}
\providecommand{\R}{{ R}}
\providecommand{\T}{{ T}}

\providecommand{\Aw}{\inns[w]{A}}
\providecommand{\Bw}{\ensuremath{{\outs[w]{B}}}}
\providecommand{\wA}{\outs[w]{A}}

\providecommand{\Tw}{\inns[-w]{T}}
\providecommand{\wT}{\outs[-w]{T}}

\providecommand{\cran}{\overline{\ran}}

\providecommand{\RR}{\mathbb{R}}

\allowdisplaybreaks 

\begin{document}
\tikzstyle{decision} = [diamond, draw, fill=blue!50]
\tikzstyle{line} = [draw, -stealth, thick]
\tikzstyle{elli}=[draw, ellipse, fill=red!50,minimum height=8mm, text width=5em, text centered]
\tikzstyle{block} = [draw, rectangle, fill=blue!50, text width=8em, text centered, minimum height=15mm, node distance=10em]

\title{\textsc{Generalized solutions for the sum of two maximally monotone
operators}}

\author{
Heinz H.\ Bauschke\thanks{
Mathematics, University
of British Columbia,
Kelowna, B.C.\ V1V~1V7, Canada. E-mail:
\texttt{heinz.bauschke@ubc.ca}.},
Warren L.\ Hare\thanks{
Mathematics,
University of British Columbia,
Kelowna, B.C.\ V1V~1V7, Canada.
E-mail:  \texttt{warren.hare@ubc.ca}.},
~and Walaa M.\ Moursi\thanks{
Mathematics, University of
British Columbia,
Kelowna, B.C.\ V1V~1V7, Canada. E-mail:
\texttt{walaa.moursi@ubc.ca}.}}

\date{June 7, 2013}
\maketitle

\vskip 8mm

\begin{abstract} \noindent
A common theme in mathematics is to define generalized solutions
to deal with problems that potentially do not have solutions. 
A classical
example is the introduction of least squares solutions via the
normal equations associated with a possibly infeasible system of
linear equations. 

In this paper, we introduce a ``normal
problem'' associated with finding a zero
of the sum of two maximally monotone operators. 
If the original problem admits solutions, then the normal problem
returns this same set of solutions. 
The normal problem may yield solutions when the original problem does not
admit any; furthermore, it has attractive variational and duality
properties. Several examples illustrate our theory. 
\end{abstract}

{\small

}

\section{Motivation and Introduction}

\subsection{A motivation from Linear Algebra}

\label{s:motLA}

A classical problem rooted in Linear Algebra and of central importance in the natural
sciences is to solve a system of linear equations, say 
\begin{equation}
\label{e:0602a}
Ax=b.
\end{equation}
However, it may occur (due to noisy data, for instance) that
\eqref{e:0602a} does not have a solution. 
An ingenious approach to cope with this situation, dating back to Carl Friedrich
Gauss and his famous prediction of the asteroid Ceres
(see, e.g., \cite[Subsection~1.1.1]{Bjorck} and \cite[Epilogue in
 Section~4.6]{Meyer}) in 1801, is  to consider the 
\emph{normal equation} associated with \eqref{e:0602a}, namely
\begin{equation}
\label{e:0602b}
A^*Ax = A^*b,
\end{equation}
where $A^*$ denotes the transpose of $A$. 
The normal equation \eqref{e:0602b} has extremely useful properties:
\begin{itemize}
\item
If the original system \eqref{e:0602a} has a solution,
then so does the associated system \eqref{e:0602b}; furthermore,
the sets of solutions of these two systems coincide in this case.
\item 
The associated system \eqref{e:0602b} always has a solution.
\item 
The solutions of the normal equations have a \emph{variational
interpretation} as \emph{least squares solutions}: they are the minimizers 
of the function $x\mapsto \|Ax-b\|^2$.
\end{itemize}
Our goal in this paper is to introduce a ``normal problem'' associated with
the problem of finding a zero of the sum of two monotone operators.
The solutions of this normal problem will agree with the solutions of the
original problem provided the latter set is nonempty. 
The normal problem will also have a variational interpretation as well as attractive
duality properties. We start developing the framework required to explain this
in the following subsection.

\subsection{The sum problem and Attouch--Th\'era duality}

Throughout this paper,
\begin{empheq}[box=\mybluebox]{equation}
\text{$X$ is
a real Hilbert space with inner
product $\scal{\cdot}{\cdot}$ }
\end{empheq}
and induced norm $\|\cdot\|$.
Recall that a set-valued operator $A\colon X\To X$ 
(i.e., $(\forall x\in X)$ $Ax\subseteq X$)
is \emph{monotone} if
$(\forall (x,x^*)\in\gr A)(\forall (y,y^*)\in\gr A)$
$\scal{x-y}{x^*-y^*} \geq 0$; $A$ is \emph{maximally monotone} if $A$ is monotone
and it is impossible to extend $A$ while keeping monotonicity.
Since subdifferential operators of 
proper lower semicontinuous convex functions
are maximally monotone, as are continuous linear operators with a
monotone symmetric part, it is not surprising that 
maximally monotone operators play an important role 
in modern optimization and variational analysis.
For relevant books on monotone operator theory and convex analysis
we refer the reader to, e.g., 
\cite{BC2011},
\cite{BorVanBook}, 
\cite{Brezis},
\cite{BurIus},
\cite{Rock70},
\cite{Rock98},
\cite{Simons1},
\cite{Simons2},
\cite{Zalinescu},
\cite{Zeidler2a},
\cite{Zeidler2b},
and \cite{Zeidler1}.
From now on, we assume that 
\begin{empheq}[box=\mybluebox]{equation}
\text{$A$ and $B$ are maximally monotone operators on $X$.}
\end{empheq}
Because it encompasses the problem of finding solutions to constrained
convex optimization problems, 
a key problem in monotone operator theory is to find
a zero of the sum $A+B$. Let us formalize this now.

\begin{definition}[primal problem]
The \emph{primal problem} associated with the (ordered) pair $(A,B)$ is to
determine the set of zeros of the sum,
\begin{equation}
Z_{(A,B)} := (A+B)^{-1}(0),
\end{equation}
also referred to as the set of \emph{primal solutions}. 
When there is no cause for confusion, we will write $Z$ instead of
$Z_{(A,B)}$. 
\end{definition}
Since addition is commutative,
it is clear that 
the order of the operators $A$ and $B$ is irrelevant and thus 
$Z_{(A,B)} = Z_{(B,A)}$.
In contrast, the order for the dual problem matters. Before we 
formally define the dual problem, we must introduce some notation.
First, 
\begin{equation}
A^\ovee := (-\Id)\circ A \circ(-\Id).  
\end{equation}		 
Note that $A^\ovee$ is also maximally monotone as 
is $(A^{-1})^\ovee = (A^{\ovee})^{-1}$, which motivates the
definition\footnote{This is similar to the notation $A^{-T}$ for
the transpose of the inverse of an invertible matrix in Linear Algebra.}
\begin{equation}\label{e:commute}
A^{-\ovee} := \big(A^{-1}\big)^\ovee = \big(A^{\ovee}\big)^{-1}.
\end{equation}

\begin{definition}[dual pair and (Attouch--Th\'era) dual problem]
\label{d:ATdual}
The \emph{dual pair} of $(A,B)$ is $(A,B)^* := (A^{-\ovee},B^{-1})$. 
The \emph{(Attouch--Th\'era) dual problem} associated with the 
pair $(A,B)$ is to 
determine the set of zeros of the sum,
\begin{equation}
K_{(A,B)} := \big(A^{-\ovee}+B^{-1}\big)^{-1}(0),
\end{equation}
also referred to as the set of \emph{dual solutions}. 
When there is no cause for confusion, we will write $K$ instead of
$K_{(A,B)}$. 
\end{definition}

This duality, pioneered by Attouch and Th\'era \cite{AT},
has very attractive properties, including the following:
\begin{itemize}
\item $(A,B)^{**} = (A,B)$.
\item The dual problem of $(A,B)$ is precisely the primal problem of
$(A,B)^*$.
\item The set of primal solutions is nonempty if and only if the set of dual
solutions is nonempty. 
\end{itemize}

\subsection{Aim of this paper}

Not every sum problem admits a solution: suppose that 
$A= N_U$ and $B=N_V$, where $U$ and $V$ are nonempty closed convex
subsets of $X$. It is clear that $Z$, the set of primal solutions
associated with $(A,B)$, is equal to $U\cap V$ --- 
\emph{however, this intersection may be empty} in which case the primal
problem does not have any solution. 

\emph{Our aim in this paper is to define a \emph{normal problem} associated
with the original sum problem with attractive and useful properties.}
Similarly to the complete extension of
classical linear equations via normal equations (see
Section~\ref{s:motLA}), our proposed approach achieves the following:
\begin{itemize}
\item
If the original problem has a solution, then so does the normal
problem and
the sets of solutions to these problems coincide. 
\item
The normal problem may have a solution even if the original
problem does not have any.
\item 
The solutions of the normal problem have a variational interpretation as 
\emph{infimal displacement} solutions related to the Douglas--Rachford splitting
operator. 
\item 
The normal problem interacts well with Attouch--Th\'era duality.
\end{itemize}

Due to some technical results that need to be reviewed and developed, 
we postpone the actual derivation and definition of the normal problem
until Section~\ref{s:normprob}. 
We conclude this introductory section with
some comments on the organization and notation of this paper. 

\subsection{Organization of the paper}

The remainder of the paper is organized as follows.
In Section~\ref{s:aux}, 
we review Attouch--Th\'era duality (Section~\ref{s:AT}), 
firmly nonexpansive operators and resolvents
(Section~\ref{s:fnos}), 
the Douglas--Rachford splitting operator (Section~\ref{s:D-R}),
and we also provide some auxiliary results on perturbations
(Section~\ref{s:pert}). 
Our main results are in Section~\ref{s:main}. 
The normal problem is introduced in 
Section~\ref{s:normprob}, after presenting
results on perturbation duality (Section~\ref{s:pd}).
Examples and directions for future research are
discussed in Section~\ref{s:ex} and \ref{s:fut}, respectively.

Throughout, we utilize standard notation from convex analysis and
monotone operator theory (see, e.g., \cite{BC2011},
\cite{Rock70}, \cite{Rock98}, or \cite{Zalinescu}). 

\section{Auxiliary results}
\label{s:aux}

\subsection{Solution mappings for Attouch--Th\'era duality}
\label{s:AT}

\begin{definition}[solution mappings]
The \emph{dual and primal solution mappings} associated with $(A,B)$ are 
\begin{equation}
\bK \colon X\To X\colon
x\mapsto (-Ax)\cap(Bx)
\end{equation}
and
\begin{equation}
\bZ \colon X\To X\colon
x\mapsto (-A^{-\ovee}x)\cap (B^{-1}x),
\end{equation}
respectively.
\end{definition}
Note that the primal solution  mapping $\bZ$ of $(A,B)$ is the dual solution
mapping of $(A,B)^*$ and analogously for $\bK$. 
The importance of these mappings stems from the following result, which
shows that the solutions mappings relate the sets of solutions $Z$ and $K$
to each other:

\begin{fact}
{\rm (See \cite{AT} or \cite[Proposition~3.1]{BHM}.)}
$\dom \bK = Z$, $\ran\bK=K$,
$\dom \bZ = K$, $\ran\bZ=Z$, and $\bZ = \bK^{-1}$. 
\end{fact}

\subsection{Firmly nonexpansive operators and resolvents}

\label{s:fnos}

Most of the material in this section is standard.
Facts without explicit references may be found in, e.g., 
\cite{BC2011}, \cite{GK}, or \cite{GR}.

\begin{definition}
Let $T\colon X\to X$. 
Then $T$ is \emph{nonexpansive}, if
\begin{equation}
(\forall x\in X)(\forall y\in X)\quad
\|Tx-Ty\|\leq \|x-y\|.
\end{equation}
Furthermore, $T$ is \emph{firmly nonexpansive}
if 
\begin{equation}
(\forall x\in X)(\forall y\in X)\quad
\|Tx-Ty\|^2 + \|(\Id-T)x-(\Id-T)y\|^2\leq \|x-y\|^2.
\end{equation}
\end{definition}

Clearly, every firmly nonexpansive mapping is nonexpansive.

\begin{fact}
Let $T\colon X\to X$.
Then $T$ is firmly nonexpansive if and only if $2T-\Id$ is nonexpansive.
\end{fact}

\begin{fact}[infimal displacement vector]
{\rm (See, e.g., \cite{BBR}, \cite{BrRe}, and \cite{Pazy}.)}
\label{f:idv}
Let $T\colon X\to X$ be nonexpansive.
Then $\cran(\Id-T)$ is convex;
consequently, the \emph{infimal displacement vector} 
$v := P_{\cran(\Id-T)}0$ is the unique element in 
$\cran(\Id-T)$ such that $(\forall x\in X)$ $\|v\|\leq \|x-Tx\|$.
\end{fact}

\begin{lemma}
\label{l:idvs}
Let $T_1\colon X\to X$ and $T_2\colon X\to X$ be nonexpansive.
Set $v_1 := P_{\cran(\Id-T_1T_2)}0$ and
$v_2 := P_{\cran(\Id-T_2T_1)}0$.
Then $\|v_1\|=\|v_2\|$.
\end{lemma}
\begin{proof}
By definition of $v_1$, there
exists a sequence $(x_n)_\nnn$ in $X$ such that
$\|x_n-T_1T_2x_n\|\to\|v_1\|$. 
Hence $(\forall\nnn)$ 
$\|v_2\|\leq\|(T_2x_n)-T_2T_1(T_2x_n)\|
\leq\|x_n-T_1T_2x_n\|$ and thus
$\|v_2\|\leq\|v_1\|$.
We see analogously that $\|v_1\|\leq\|v_2\|$.
\end{proof}

\begin{definition}[resolvent and reflected resolvent]
The \emph{resolvent} of $A$ is the operator 
\begin{equation}
J_A := (\Id+A)^{-1},
\end{equation}
and the \emph{reflected resolvent} is
\begin{equation}
R_A := 2J_A-\Id.
\end{equation}
\end{definition}

\begin{fact}
$J_A$ is firmly nonexpansive and $R_A$ is nonexpansive.
Furthermore,
\begin{equation}
\label{e:invresid}
J_A + J_{A^{-1}} = \Id.
\end{equation}
\end{fact}

\begin{example}
Let $U$ be a nonempty closed convex subset of $X$,
and suppose that $A=N_U$ is the corresponding normal cone
operator. Then $J_A=P_U$ is the projection operator onto $U$
and $R_A=2P_U-\Id$ is the corresponding reflector.
\end{example}

\begin{proposition}
\label{p:0603a}
Suppose that $A\colon X\to X$ is continuous, linear, and single-valued
such that $A$ and $-A$ are monotone, and $A^2 = -\alpha\Id$, where $\alpha\in\RP$.
Then 
\begin{equation}
J_A = \frac{1}{1+\alpha}\big(\Id-A\big)
\;\;\text{and}\;\;
R_A = \frac{1-\alpha}{1+\alpha}\Id - \frac{2}{1+\alpha}A.
\end{equation}
\end{proposition}
\begin{proof}
We have
\begin{subequations}
\begin{align}
J_AJ_{-A}&=(\Id+A)^{-1}(\Id-A)^{-1}
=\big((\Id-A)(\Id+A)\big)^{-1}
=(\Id-A^2)^{-1}
=(\Id+\alpha\Id)^{-1}\\
&=\frac{1}{1+\alpha}\Id.
\end{align}
\end{subequations}
It follows that
$J_A = (1+\alpha)^{-1}J_{-A}^{-1}=(1+\alpha)^{-1}(\Id-A)$ and hence that
\begin{equation}
R_A = 2J_A-\Id = \frac{2}{1+\alpha}(\Id-A)-\Id
= \frac{1-\alpha}{1+\alpha}\Id - \frac{2}{1+\alpha}A,
\end{equation}
as claimed.
\end{proof}

\begin{example}
\label{ex:sushirot}
Suppose that $X=\RR^2$ and that $A\colon\RR^2\to\RR^2\colon (x,y)\mapsto
(-y,x)$ is the rotator by $\pi/2$. Then
$A^2=-\Id$; consequently, by Proposition~\ref{p:0603a},
$J_A=(1/2)(\Id-A)$ and $R_A = -A$.
\end{example}

\subsection{The Douglas--Rachford splitting operator}

\label{s:D-R}

\begin{definition}
\label{d:DR}
The \emph{Douglas--Rachford splitting operator} associated with $(A,B)$ is 
\begin{empheq}[box=\mybluebox]{equation}
T := T_{(A,B)} := J_AR_B+\Id-J_B.
\end{empheq}
We will simply use $T$ instead of $T_{A,B}$ provided there is no
cause for confusion. 
\end{definition}

\begin{fact}
\label{f:DR}
The following hold:
\begin{enumerate}
\item
\label{f:DRi}
$T_{(A,B)} = \thalb(\Id+R_AR_B)$; consequently, $T_{(A,B)}$ is firmly
nonexpansive.
\item 
\label{f:DRii}
\emph{\textbf{(Eckstein)} (See \cite[Lemma~3.6]{EckThesis}.)}
$T_{(A,B)} = T_{(A,B)^*} = T_{(A^{-\ovee},B^{-1})}$.
\item \emph{\textbf{(Eckstein)} (See \cite[Proposition~4.1]{EckThesis}.)}
\begin{equation}
\gr(T)=\menge{(b+b^*,a+b^*)}{(a,a^*)\in \gr A,(b,b^*)\in \gr B
,b-a=b^*+a^* }.
\end{equation}
\end{enumerate}
\end{fact}

\begin{corollary}
We have 
\begin{equation}
\gr(\Id-T)=\menge{(b+b^*,b-a )}{(a,a^*)\in \gr A,(b,b^*)\in \gr B
,b-a=b^*+a^* };
\end{equation}
consequently, 
\begin{subequations}
\label{con:eq}
\begin{align}
\ran(\Id -T)&=\menge{b-a}{(a,a^*)\in \gr A,~(b,b^*)\in \gr B ,~
b-a=b^*+a^*}\\
&\subseteq (\dom B-\dom A)\cap(\ran A + \ran B).
\end{align}
\end{subequations}
\end{corollary}

It is clear from the definition that and
Fact~\ref{f:DR}\ref{f:DRi} that $\Id-T_{A,B}$ is also firmly
nonexpansive. In fact, we note in passing that 
$\Id-T_{A,B}$ is itself a 
Douglas--Rachford splitting operator:

\begin{proposition}
$\Id-T_{(A,B)} = T_{(A^{-1},B)}$.
\end{proposition}
\begin{proof}
Using \eqref{e:invresid}, we obtain
\begin{subequations}
\begin{align}
T_{A,B}+T_{A^{-1},B}&= \Id -J_B+
J_AR_B+ \Id -J_B+J_{A^{-1}}R_B\\
&=2\Id-2J_{B}+(J_A+J_{A^{-1}})R_B\\
&=2\Id-2J_B + R_B\\
&=\Id,
\end{align}
\end{subequations}
and the conclusion follows. 
\end{proof}

\begin{fact}
\label{f:Psi}
{\rm (See \cite[Theorem~4.5]{BHM}.)}
The mapping 
\begin{equation}
\Psi\colon\gr\bK\to\Fix T\colon (z,k)\mapsto z+k
\end{equation}
is a well defined bijection that is continuous in both directions,
with $\Psi^{-1}\colon x\mapsto (J_Bx,x-J_Bx)$. 
\end{fact}

\begin{corollary}[Combettes]
\label{c:PLC}
{\rm (See \cite[Lemma~20.6(iii)]{Comb04}.)}
$J_B(\Fix T)=Z$.
\end{corollary}

\subsection{Perturbation calculus}

\label{s:pert}

\begin{definition}[shift operator and corresponding inner/outer perturbations]
Let $w\in X$. We define the associated \emph{shift operator}
\begin{equation}
S_w \colon X\to X \colon x\mapsto x-w,
\end{equation}
and we extend $S_w$ to deal with subsets of $X$ by setting
$(\forall C\subseteq X)$ $S_w(C) := \bigcup_{c\in C} \{S_w(c)\}$. 
We define the corresponding \emph{inner and outer perturbations} of $A$ by
\begin{equation}
\inns[w]{A} := A\circ S_w\colon X\To X\colon x\to A(x-w),\label{w:in} 
\end{equation}
and
\begin{equation}
\outs[w]{A}:=S_w\circ A\colon X\To X\colon x\to Ax-w.\label{w:out}
\end{equation}
\end{definition}

Observe that if $w\in X$, then 
the operators $\Aw$ and $\wA$ are maximally monotone,
with domains $S_{-w}(\dom A) = w+\dom A$ and $\dom A$, respectively.

\begin{lemma}[perturbation calculus]\label{oper:1}
Let $w\in X$. 
Then the following hold:
\begin{enumerate}
\item 
\label{oper:1i}
 $\inns[w]{A}^{-1}=\outs[-w]{A^{-1}}$.
\item
\label{oper:1ii}
$\outs[w]{A}^{-1}=\inns[-w]{A^{-1}}$.
\item 
\label{oper:1iii}
$\inns[w]{A}^{\ovee}=\inns[-w]{A^{\ovee}}$.
\item 
\label{oper:1iv}
$\outs[w]{A}^{\ovee}=\outs[-w]{A^{\ovee}}$.
\item 
\label{oper:1v}
$\inns[w]{A}^{-\ovee}=\outs[w]{A^{-\ovee}}$.
\item 
\label{oper:1vi}
$\outs[w]{A}^{-\ovee}=\inns[w]{A^{-\ovee}}$. 
\end{enumerate}
\end{lemma}
\begin{proof}
Let $(x,y)\in X^2$. 
\ref{oper:1i}:
$y\in \inns{A}^{-1}x$ $\iff$ $x\in \inns{A}y = A(y-w)$ 
$\iff$ $y-w\in A^{-1} x$ $\iff$ $y \in A^{-1}x+w=\outs[-w]{A^{-1}}x$.
\ref{oper:1ii}:
$y\in \outs{A}^{-1}x$ $\iff$ $x\in \outs{A}y$ $\iff$ $x \in A y -w$ 
 $\iff$ $x+w\in Ay$ $\iff$ $y\in A^{-1}(x+w)=\inns[-w]{A^{-1}}x$.
\ref{oper:1iii}:
$\Aw^{\ovee} = -\Aw(-x)=-A(-x-w) = A^\ovee(x+w) = \inns[-w]{A^\ovee}$. 
\ref{oper:1iv}:
$\wA^\ovee x  = -\wA(-x) = -(A(-x)-w) = A^\ovee x-(-w) = \outs[-w]{A^\ovee}x$. 
\ref{oper:1v}:
Using \ref{oper:1i} and \ref{oper:1iv}, we see that 
$\Aw^{-\ovee} = (\Aw^{-1})^\ovee = \outs[-w]{A^{-1}}^\ovee =
\outs[w]{A^{-\ovee}}$. 
\ref{oper:1vi}:
Using \ref{oper:1ii} and \ref{oper:1iii}, we see that 
$\wA^{-\ovee} = (\wA^{-1})^\ovee = \inns[-w]{A^{-1}}^\ovee =
\inns[w]{A^{-\ovee}}$. 
\end{proof} 

As an application, we record the following result which will be useful
later.

\begin{corollary}[dual of inner-outer perturbation]
Let $w\in X$. Then 
\begin{equation}
\big(\Aw,\Bw\big)^{*}=
\big(
\Aw^{-\ovee},\Bw^{-1}\big)=\big(\outs{A^{-\ovee}},\inns[-w]{B^{-1}}\big). 
\end{equation}
\end{corollary}
\begin{proof}
Combine Definition~\ref{d:ATdual}
with 
Lemma~\ref{oper:1}\ref{oper:1v}\&\ref{oper:1ii}. 
\end{proof}

\subsection{Perturbations of the Douglas--Rachford operator}

We now turn to the Douglas--Rachford operator.

\begin{proposition}
\label{fix:1}
Let $w\in X$. Then the following hold:
\begin{enumerate}
\item
\label{fix:1i}
If $x\in\fix\wT$, then 
$x-w-J_Bx\in \outs{B}J_Bx\cap(-\inns{A}J_Bx)$.
\item
\label{fix:1ii}
If $y\in \outs{B}z\cap(-\inns{A}z)$, then $x=w+y+z\in\Fix\wT$ and $z=J_Bx$.
\end{enumerate}
\end{proposition}
 \begin{proof}
If $x\in X$, then $x-w-J_Bx\in\outs[w]{B}J_Bx$. 

\ref{fix:1i}: 
Since $x\in \fix(\wT)$, we have $x-\T x=w$;
equivalently, 
$J_Bx-w=J_A R_Bx$.
Hence $2J_Bx-x=R_Bx\in(A+\Id)(J_Bx-w)=\Aw J_Bx + J_Bx-w$
and thus $-(x-w-J_Bx)\in\Aw J_Bx$.

\ref{fix:1ii}: 
Since $y\in \outs{B}z\cap(-\inns{A}z) = (Bz-w)\cap(-A(z-w))$,
we have $z=J_Bx$ and $z-w=J_A(-y+z-w)$.
Hence $R_Bx = 2J_Bx-x=2z-(w+y+z)=z-w-y$ and
so $J_AR_Bx = J_A(z-w-y)=z-w$.
Thus, $x-Tx=
J_Bx-J_AR_Bx=z-(z-w)=w$. 
\end{proof}

\begin{corollary}
\label{cfix1}
Let $w\in X$. Then
$\displaystyle 
\fix {\wT}=w+ \bigcup_{z\in X} \big(z+\outs{B}z\cap(-\inns{A}z)\big)$.
\end{corollary}

\begin{proposition}\label{p:general}
Let $w\in X$. Then 
\begin{equation}
\label{e:0602c}
\T_{(\Aw,\Bw)}=\Tw
\end{equation}
and
\begin{equation}
\label{e:0602d}
\fix {\Tw}= -w+\fix\outs[-w]{T}=\bigcup_{z\in X} \Big(z+\big((Bz-w)\cap(-A(z-w))\big)\Big).
\end{equation}
\end{proposition}
\begin{proof}
Let $x\in X$. 
Using, e.g., \cite[Proposition 23.15]{BC2011}, we obtain
$\J_{\Aw}x=\J_A(x- w)+w$ and $\J_{\Bw}x=\J_B(x+w)$. 
Consequently, 
$\R_{\Aw}x=2\J_A(x-w)+2w-x$ and $\R_{\Bw}x=2\J_B(x+w)-x$. 
It thus follows with Definition~\ref{d:DR} that 
\begin{subequations}
\begin{align}
\T_{(\Aw,~\Bw)} x & = x -\J_{\Bw }x+\J_{\Aw} \R_{\Bw} x\\
&= x - \J_B(x+w)+\J_A\big(2\J_B(x+w)-x-w\big)+w\\
&=(x+w)-\J_B(x+w)+\J_A\big(\R_B(x+w)\big)\\
&=\T(x+w)=\Tw x,
\end{align} 
\end{subequations}
and so \eqref{e:0602c} holds. 
Next, $x\in\Fix\Tw$ 
$\iff$ $x=T(x+w)$ $\iff$ $x+w=w+T(x+w)$
$\iff$ $x+w\in\Fix\outs[-w]{T}$, and have thus verified the left identity in
\eqref{e:0602d}. 
To see the right identity in \eqref{e:0602d}, use Corollary~\ref{cfix1}. 
\end{proof}

We now obtain a generalization of Fact~\ref{f:Psi}, which corresponds to
the case when $w=0$. 

\begin{proposition}
Let $w\in X$ and define
\begin{equation}
\bK_w\colon X\To X\colon
x\mapsto (-A(x-w))\cap(Bx-w).
\end{equation}
Then 
\begin{equation}
\Psi_w\colon\gr\bK_w\to\Fix \wT \colon (z,k)\mapsto z+k+w
\end{equation}
is a well defined bijection that is continuous in both directions,
with $\Psi_w^{-1}\colon x\mapsto (J_Bx,x-J_Bx-w)$. 
\end{proposition}
\begin{proof}
For the pair $(\Aw,\Bw)$, 
the dual solution mapping is $\bK_w$ and 
the Douglas--Rachford operator is $\Tw$ by \eqref{e:0602c}.
Applying Fact~\ref{f:Psi} in this context, we obtain
\begin{equation}
\Phi\colon \gr\bK_w\to \Fix\Tw \colon(z,k)\mapsto z+k
\end{equation}
is continuous in both directions with $\Phi^{-1}
\colon x\mapsto (J_{\Bw}x,x-J_{\Bw}x) = (J_B(x+w),x-J_B(x+w))$.
Furthermore, $S_{-w}$ is a bijection from $\Fix\Tw$ to $\Fix \wT$
by \eqref{e:0602d}. 
This shows that $\Psi_w = S_{-w}\circ \Phi$ and the result follows. 
\end{proof}

\section{The normal problem}

\label{s:main}

\subsection{The $w$-perturbed problem}

\label{s:pd}

\begin{definition}[$w$-perturbed problem]
Let $w\in X$.
The \emph{$w$-perturbation} of $(A,B)$ is 
$(\Aw,\Bw)$. 
The \emph{$w$-perturbed problem} 
associated with the pair $(A,B)$ is to
determine the set of zeros
\begin{equation}
Z_w := Z_{(\Aw,\Bw)} = \big(\Aw+\Bw\big)^{-1}(0).
\end{equation}
\end{definition}
Note that the $w$-perturbed problem of $(A,B)$ is precisely
the primal problem of $(\Aw,\Bw)$, i.e., of the $w$-perturbation
of $(A,B)$.

\begin{proposition}[Douglas--Rachford operator of the
$w$-perturbation]
\label{p:0602e}
Let $w\in X$. 
Then the Douglas--Rachford operator of the $w$-perturbation 
$(\Aw,\Bw)$ of $(A,B)$ is
\begin{equation}
T_{(\Aw,\Bw)} = \Tw.
\end{equation}
\end{proposition}
\begin{proof}
This follows from \eqref{e:0602c} of Proposition~\ref{p:general}. 
\end{proof}

\begin{proposition}
\label{p:0602f}
Let $w\in X$.
Then 
\begin{equation}
\label{e:0604z}
Z_w = J_\Bw\big(\Fix\Tw\big) = J_B\big(w+\menge{x\in X}{x=T(x+w)}\big).
\end{equation}
Furthermore, the following are equivalent:
\begin{enumerate}
\item
\label{p:0602fi}
$Z_w\neq\varnothing$.
\item
\label{p:0602fii}
$\Fix\Tw \neq\varnothing$.
\item
\label{p:0602fiii}
$w\in\ran(\Id-T)$.
\item
\label{p:0602fiv}
$w\in\ran(\Aw+B)$.
\end{enumerate}
\end{proposition}
\begin{proof}
The identity \eqref{e:0604z} follows by
combining Corollary~\ref{c:PLC} with Proposition~\ref{p:0602e}.
This also yields the equivalence of 
\ref{p:0602fi} and \ref{p:0602fii}. 
Let $x\in X$.
Then
$x\in Z_w$
$\iff$
$0\in \Aw x+\Bw x$
$\iff$
$w\in \Aw x + Bx$,
and we deduce the equivalence of \ref{p:0602fi} and
\ref{p:0602fiv}. 
Finally,
$x\in\Fix\Tw$
$\iff$
$x=T(x+w)$
$\iff$
$w\in (\Id-T)(x+w)$,
which yields the equivalence of
\ref{p:0602fii} and \ref{p:0602fiii}.
\end{proof}

The equivalence of \ref{p:0602fi} and \ref{p:0602fiii}
yields the following key result on which
$w$-perturbations have nonempty solution sets.

\begin{corollary}
\label{c:tick}
$\menge{w\in X}{Z_w\neq\varnothing} = \ran(\Id-T)$.
\end{corollary}

\begin{remark}[Attouch--Th\'era dual of the perturbed problem]
\label{r:ATcompnp}
Consider the given pair of monotone operators $(A,B)$.
We could either first perturb and then take the Attouch--Th\'era dual or
start with the Attouch--Th\'era dual and then perturb.
It turns out that the order of these operations does not matter --- up to a
horizontal shift of the graphs. Indeed, 
for every $x\in X$, we have
\begin{subequations}
\begin{align}
\big(\inns[w]{A^{-\ovee}}+ \outs[w]{B^{-1}}\big)x
&=A^{-\ovee}(x-w)+B^{-1}x-w\\
&= A^{-\ovee}(x-w)-w+B^{-1}((x-w)+w)\\
&=\outs[w]{A^{-\ovee}}(x-w)+\inns[-w]{B^{-1}}(x-w)\\
&=\big(\outs[w]{A^{-\ovee}}+ \inns[-w]{B^{-1}}\big)(x-w).
\end{align}
\end{subequations}
Hence $\gr(\inns[w]{A^{-\ovee}}+ \outs[w]{B^{-1}}) = 
(w,0)+\gr (\outs[w]{A^{-\ovee}}+ \inns[-w]{B^{-1}})$, 
which gives rise to the following diagram:
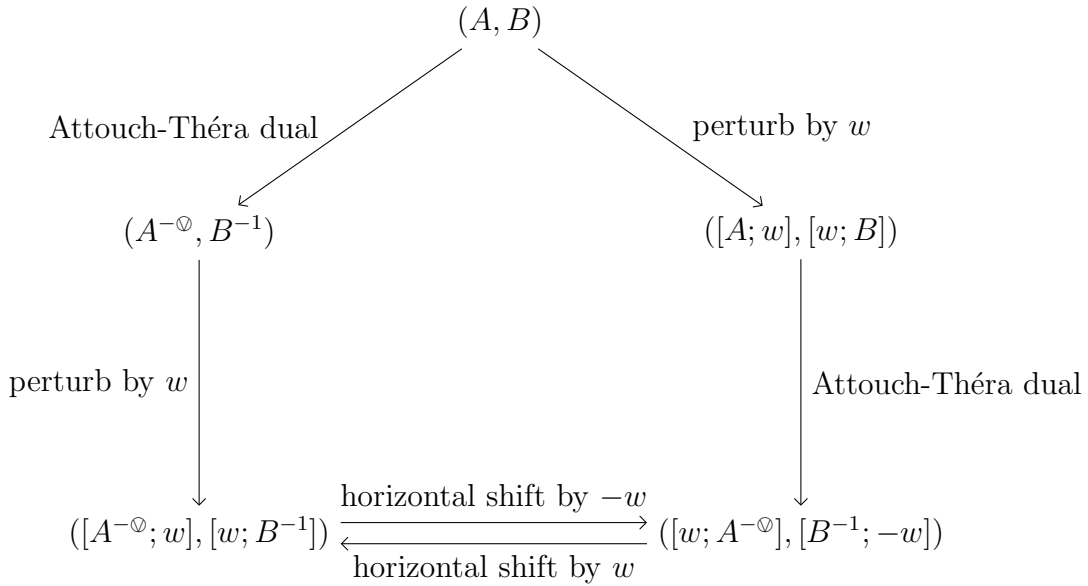
\begin{center}
\begin{tikzpicture}[scale=4]
\node (P0) at (0,-0.3) {$(A,B)$};
\node (P1)
at (-1,-1) {$(A^{-\ovee},B^{-1})$} ;
\node (P2) at (-1,-2)
{$(\inns[w]{A^{-\ovee}},\outs[w]{B^{-1}})$};
\node (P3) at (1,-2)
{$(\outs[w]{A^{-\ovee}},\inns[-w]{B^{-1}})$};
\node (P4) at (1,-1)
{$(\inns[w]{A},\outs[w]{B})$};
\draw
(P0) edge[->,>=angle 90] node[left]
{Attouch-Th\'{e}ra dual~~~} (P1)
(P1) edge[->,>=angle 90] node[left] { ~~perturb by $w$} (P2)
([yshift= 1pt]P2.east) edge[->,>=angle 90] node[above] 
{horizontal shift by $-w$} ([yshift= 1pt]P3.west)
([yshift= -1pt]P2.east) edge[<-,>=angle 90] node[below] 
{horizontal shift by $w$} ([yshift= -1pt]P3.west)
(P4) edge[->,>=angle
90] node[right] {Attouch-Th\'{e}ra dual~} (P3)
(P0) edge[->,>=angle 90]
node[right] { ~~~perturb by $w$} (P4);
\end{tikzpicture}
\end{center}
\end{remark}

\subsection{The normal problem}

\label{s:normprob}

We are now in a position to define the normal problem.

\begin{definition}[infimal displacement vector and the normal problem]
\label{d:normal}
The vector 
\begin{equation}
v(A,B) = P_{\cran(\Id-T)}0
\end{equation}
is the \emph{infimal displacement vector of $(A,B)$}. 
The \emph{normal problem} associated with $(A,B)$
is the $v(A,B)$-perturbed problem of $(A,B)$, and
the set of \emph{normal solutions} is $Z_{v(A,B)}$. 
\end{definition}

\begin{remark}[new notions are well defined]
The notions presented in Definition~\ref{d:normal}
are \emph{well defined}: indeed,
since $T$ is firmly nonexpansive (Fact~\ref{f:DR}\ref{f:DRi}),
it is also nonexpansive
and the existence and uniqueness 
of $v(A,B)$ follows from Fact~\ref{f:idv}.
\end{remark}

\begin{remark}[new notions extend original notions]
Suppose that for the original problem $(A,B)$, we have
$Z = Z_0 = (A+B)^{-1}(0)\neq\varnothing$. 
By Corollary~\ref{c:tick}, $0\in\ran(\Id-T)$ and so
$v(A,B)=0$. Hence the normal problem coincides with the original
problem, as do the associated sets of solutions. 
\end{remark}

\begin{remark}[normal problem may or may not have solutions]
If the set of original solutions $Z$ is empty,
then the set of normal solutions may be 
either nonempty (see Example~\ref{ex:noyes}) or empty
(see Example~\ref{ex:nono}). 
\end{remark}

The original problem of finding a zero of $A+B$ is clearly
symmetric in $A$ and $B$. We now present a statement about the
\emph{magnitude} of the corresponding infimal displacement
vectors:

\begin{proposition}
$\|v(A,B)\|=\|v(B,A)\|$.
\end{proposition}
\begin{proof}
It follows from Fact~\ref{f:DR}\ref{f:DRi} that
\begin{equation}
\Id-T_{(A,B)} = \thalb(\Id-R_AR_B)
\;\;\text{and}\;\;
\Id-T_{(B,A)} = \thalb(\Id-R_BR_A).
\end{equation}
Thus, using Lemma~\ref{l:idvs}, we see that 
$\|v(A,B)\|=2 \|P_{\cran(\Id-R_AR_B)}0\| = 
2\|P_{\cran(\Id-R_AR_B)}0\|=\|v(B,A)\|$.
\end{proof}

\begin{remark}[$v(A,B)\neq v(B,A)$ may occur]
We will see in the sequel examples where
$v(A,B)\neq 0$ but
(i) $v(B,A)=-v(A,B)$ (see Remark~\ref{r:-v});
(ii) $v(B,A)\perp v(A,B)$ (see Example~\ref{ex:john});
or (iii) $v(A,B)=v(B,A)$ (see Example~\ref{ex:v=v}). 
\end{remark}

\begin{remark}[self-duality: $v(A,B)=v(A^{-\ovee},B^{-1})$]
Since, by Fact~\ref{f:DR}\ref{f:DRii},
$T_{(A,B)} = T_{(A^{\ovee},B^{-1})}$, 
it is clear that $v(A,B)=v(A^{-\ovee},B^{-1})$. 
It follows from Remark~\ref{r:ATcompnp} that the operations
of perturbing by $v(A,B)$ and taking the Attouch--Th\'era dual commute,
up to a shift.
\end{remark}

\subsection{Examples}

\label{s:ex}

\begin{proposition}
\label{p:0603b}
$(\Id-J_A)B^{-1}0 \subseteq \ran(\Id-T)\subseteq \dom B-\dom A$.
\end{proposition}
\begin{proof}
The right inclusion follows from \eqref{con:eq}. 
To tackle the left inclusion, 
suppose that $z\in B^{-1}0$ and set $w := z-J_Az$.
Then $w=z-J_Az\in A(J_Az)=A(z-w)+0\subseteq \Aw(z)+Bz$.
Hence, by Proposition~\ref{p:0602f}, 
$w\in\ran(\Id-T)$. 
\end{proof}

\begin{proposition}[normal cone operators]
\label{p:0603c}
Suppose that $A= N_U$ and $B=N_V$,
where $U$ and $V$ are nonempty closed convex subsets of $X$.
Then
\begin{equation}
\label{e:0603d}
v{(A,B)} = P_{\overline{V-U}}0 
\end{equation}
and the set of normal solutions is
\begin{equation}
V\cap(v{(A,B)}+U)=\Fix(P_VP_U).
\end{equation}
\end{proposition}
\begin{proof}
Since $B^{-1}0 = V$ and $J_A = P_U$, Proposition~\ref{p:0603b} yields
$C := \menge{v-P_Uv}{v\in V} \subseteq \ran(\Id-T) \subseteq V- U$; hence,
\begin{equation}
\overline{C} \subseteq \cran(\Id-T) \subseteq
\overline{V-U}.
\end{equation}
Set $g := P_{\overline{V-U}}0$. 
By \cite[Theorem~4.1]{BB94}, there exists a sequence $(v_n)_\nnn$ in $V$ such
that $v_n-P_Uv_n\to g$. It follows that $(v_n-P_Uv_n)_\nnn$ lies in $C$ and
hence that $g\in \overline{C}$. Therefore
$P_{\overline{C}}0 = P_{\cran(\Id-T)}0 = P_{\overline{V-U}}0$ and we obtain
\eqref{e:0603d}.
(For an alternative proof, see \cite[Theorem~3.5]{BCL04}.)

Let $x\in X$. Then $x$ is a normal solution if and only if
\begin{equation}
\label{e:0603e}
g \in N_U(x-g)+N_V(x).
\end{equation}
Assume first that \eqref{e:0603e} holds.
Then $x\in V$ and $x-g\in U$.
Hence $x\in V\cap (g+U)=\Fix(P_VP_U)$ by 
\cite[Lemma~2.2]{BB94}. Conversely, assume $x\in V\cap(g+U)=\Fix(P_VP_U)$. 
Then $x\in V$, $x-g\in U$, $P_Ux=x-g$ and $P_V(x-g)=x$.
Hence $N_U(x-g)\supseteq \RP g$ and $N_V(x)\supseteq \RM g$;
consequently, $N_U(x-g)+N_V(x)\supseteq \RP g +\RM g = \RR g\ni g$ and
therefore \eqref{e:0603e} holds.
\end{proof}

\begin{remark}
\label{r:-v}
Proposition~\ref{p:0603c} is consistent with the theory dealing with
inconsistent feasibility problems (see, e.g., \cite{BB94}).
Note that it also yields the formula
\begin{equation}
v(A,B) = -v(B,A)
\end{equation}
in this particular context. 
\end{remark}

\begin{example}[no original solutions but normal solutions exist]
\label{ex:noyes}
Suppose that $A$ and $B$ are as in Proposition~\ref{p:0603c},
that $U\cap V=\varnothing$, and $V$ is also bounded.
Then $\Fix(P_VP_U)\neq\varnothing$ by the Browder--G\"ohde--Kirk fixed
point theorem (see, e.g., \cite[Theorem~4.19]{BC2011}). 
So, the original problem has no solution but there exist normal solutions.
\end{example}

\begin{example}[neither original nor normal solutions exist]
\label{ex:nono}
Suppose that $X=\RR^2$, that 
$A$ and $B$ are as in Proposition~\ref{p:0603c},
that $U=\RR\times\{0\}$, and that
$V=\menge{(x,y)\in\RR^2}{\beta+\exp(x)\leq y}$, where $\beta\in\RP$.
Then $v{(A,B)}=(\beta,0)$ yet $\Fix(P_VP_U)=\varnothing$.
\end{example}

\begin{example}
\label{ex:john}
Suppose that $X=\RR^2$, 
let $L\colon\RR^2\to\RR^2\colon(\xi,\eta)\mapsto(-\eta,\xi)$ be the rotator
by $\pi/2$, let $a^*\in\RR^2$ and $b^*\in\RR^2$.
Suppose that $(\forall x\in\RR^2)$
$Ax=Lx+a^*$ and $Bx=-Lx-b^*$.
Now let $x\in X$ and let $w\in X$.
Then $0=A(x-w)+Bx-w = L(x-w)+a^*-Lx-b^*-w$ 
and so $(\Id+L)w=a^*-b^*$, i.e., $w=J_L(a^*-b^*) =
(1/2)(\Id-L)(a^*-b^*)$ by Example~\ref{ex:sushirot}. 
It follows that 
\begin{equation}
v{(A,B)} = \thalb(\Id-L)(a^*-b^*).
\end{equation}
An analogous argument yields
\begin{equation}
v{(B,A)} = \thalb(\Id+L)(b^*-a^*).
\end{equation}
Setting $d^*=b^*-a^*$, we have
$4\scal{v(A,B)}{v(B,A)} = \scal{Ld^*-d^*}{Ld^*+d^*}=
\|Ld^*\|^2-\|d^*\|^2 = 0$.
and $v(A,B)+v(B,A)=Ld^*$.
Thus if $d^*\neq 0$, i.e., $a^*\neq b^*$, then
\begin{equation}
v(A,B)\neq 0
\;\;\text{and}\;\; 
v(A,B)\perp v(B,A). 
\end{equation}
\end{example}

\begin{example}
\label{ex:v=v}
Suppose that there exists $a^*$ and $b^*$ in $X$ such that
$\gr A = X\times \{a^*\}$ and $\gr B = X\times \{b^*\}$.
By \eqref{con:eq},
$\varnothing\neq\ran(\Id-T)\subseteq \{a^*+b^*\}$. 
Hence $v(A,B)=a^*+b^*$ and analogously
$v(B,A)=a^*+b^*$. Thus, if $a^*+b^*\neq 0$, we have
\begin{equation}
v(A,B)\neq 0
\;\;\text{and}\;\; 
v(A,B)= v(B,A). 
\end{equation}
\end{example}

\begin{proposition}
\label{p:peepee}
Suppose that there exists 
continuous linear monotone operators $L$ and $M$ on $X$,
and vectors $a^*$ and $b^*$ in $X$ such that
$(\forall x\in X)$
$Ax=Lx+a^*$ and $Bx=Mx+b^*$.
Consider the problem
\begin{equation}
\label{e:peepee}
\text{\rm minimize $\|w\|^2$ ~~subject to~~ 
$(w,x)\in X\times X$ and $(\Id+L)w-(L+M)x=a^*+b^*$.}
\end{equation}
Let $(w,x)\in X\times X$.
Then $(w,x)$ solves \eqref{e:peepee}
$\iff$ $w=v(A,B)$ and $x$ is a normal solution
$\iff$ 
$w=P_{J_L(\ran(A+B))}0$ and $x\in(A+B)^{-1}(\Id+L)w$. 
\end{proposition}
\begin{proof}
Then $w = \Aw x + Bx$ 
$\iff$
$(\Id+L)w-(L+M)x=a^*+b^*$
$\iff$
$(\Id+L)w=(L+M)x+a^*+b^*$
$\iff$
$w = J_L\big((L+M)x+a^*+b^*\big) = J_L(A+B)x$. 
The conclusion thus follows from Proposition~\ref{p:0602f}. 
\end{proof}

It is nice to recover a special case of our original motivation given in
Section~\ref{s:motLA}:

\begin{example}[classical least squares solutions]
Suppose that $X=\RR^n$, let $M\in \RR^{n\times n}$ be such that $M+M^*$ is
positive semidefinite, and let $b\in \RR^n$. Suppose that 
$(\forall x\in \RR^n)$ $Ax=-b$ and $B=M$ so that the original problem is to
find $x\in \RR^n$ such that $Mx=b$.
Then $v(A,B)= P_{\ran M}(b)-b$ and the normal solutions are precisely the
least squares solutions. 
\end{example}
\begin{proof}
We will use Proposition~\ref{p:peepee}. 
The constraint in \eqref{e:peepee} turns into
$(\Id+0)w-(0+M)x=0+(-b)$, i.e., $w=Mx-b$ so that the optimization problem
in \eqref{e:peepee} is
\begin{equation}
\text{minimize $\|Mx-b\|^2$.}
\end{equation}
Hence the normal solutions in our sense are precisely the classical least
squares solutions. 
Furthermore,
$v(A,B)=P_{\ran(A+B)}0 = P_{-b+\ran M}(0) = P_{\ran M}(b)-b$. 
\end{proof}

\subsection{Future research}

\label{s:fut}

We conclude by outlining some research directions:

\begin{itemize}
\item Note that 
the infimal displacement vector can be found as
\begin{equation}
\label{e:pazy}
(\forall x\in X)\quad
v(A,B) = -\lim_{n\to\infty} \frac{T^nx}{n} = \lim_{n\to\infty}
T^nx-T^{n+1}x;
\end{equation}
see \cite{BBR}, \cite{BrRe}, and \cite{Pazy}. 
Conceptionally, we can thus first find $v(A,B)$ via either
iteration in \eqref{e:pazy}, and proceed then
by iterating the operator $x\mapsto T(x+v(A,B))$ to find a normal
solution. It would be desirable to devise an algorithm that
approximates $v(A,B)$ and a corresponding normal solution (should
it exist) \emph{simultaneously}. 
Proposition~\ref{p:peepee}, which leads us to solving a quadratic
optimization problem, suggests that this may indeed be possible in general.

\item Another avenue for future research is to consider more general
sums of the form $A+L^*BL$, where $L$ is a linear operator.

\item Finally, it would be interesting to relate our
perturbation technique to classical perturbation techniques
already developed for convex optimization; see, e.g., \cite{Bot}. 
\end{itemize}


\small 

\begin{thebibliography}{999}
\bibitem{AT}
H.\ Attouch and M.\ Th\'era,
A general duality principle for the sum of two operators,
\emph{Journal of Convex Analysis} 3 (1996), 1--24.

\bibitem{BBR}
J.B.\ Baillon, R.E.\ Bruck, and S.\ Reich,
On the asymptotic behavior of nonexpansive mappings
and semigroups in Banach spaces,
\emph{Houston Journal of Mathematics}~4(1) (1978), 1--9. 

\bibitem{BHM} 
H.H.\ Bauschke, R.I.\ Bo\c{t}, W.L.\ Hare, and W.M.\ Moursi,
Attouch-Th\'era 
duality revisited: paramonotonicity and operator splitting, 
\emph{Journal of Approximation Theory}~164 (2012), 1065--1084. 

\bibitem{BB94}
H.H.\ Bauschke and J.M.\ Borwein,
Dykstra's alternating projection algorithm for two sets, 
\emph{Journal of Approximation Theory}~79 (1994), 418--443. 
 
\bibitem{BC2011}
H.H.\ Bauschke and P.L.\ Combettes,
\emph{Convex Analysis and Monotone 
Operator Theory in Hilbert Spaces},
Springer, 2011.

\bibitem{BCL04}
H.H.\ Bauschke, P.L.\ Combettes,
and D.R.\ Luke,
Finding best approximation pairs relative to two
closed convex sets in Hilbert spaces,
\emph{Journal of Approximation Theory}
~127 (2004), 178--192.

\bibitem{Bjorck}
{\AA}.\ Bj{\"o}rck,
\emph{Numerical Methods for Least Squares Problems},
SIAM, 1996.

\bibitem{BorVanBook}
J.M.\ Borwein and J.D.\ Vanderwerff,
\emph{Convex Functions},
Cambridge University Press, 2010.

\bibitem{Bot} R.I.\ Bo\c{t},
\emph{Conjugate Duality in Convex Optimization},
Springer, 2010.

\bibitem{Brezis}
H. Br\'ezis,
\emph{Operateurs Maximaux Monotones et
Semi-Groupes de Contractions dans les Espaces de Hilbert},
North-Holland/Elsevier, 1973. 

\bibitem{BrRe}
R.E.\ Bruck and S.\ Reich,
Nonexpansive projections and resolvents of
accretive operators in Banach spaces,
\emph{Houston Journal of Mathematics}~3(4) (1977), 459--470.

\bibitem{BurIus}
R.S.\ Burachik and A.N.\ Iusem,
\emph{Set-Valued Mappings and Enlargements
of Monotone Operators},
Springer, 2008.

\bibitem{Comb04}
P.L.\ Combettes,
Solving monotone inclusions via compositions of nonexpansive averaged
operators,
\emph{Optimization}~53 (2004), 475--504.

\bibitem{EckThesis}
J.\ Eckstein,
\emph{Splitting Methods for Monotone Operators with
Applications to Parallel Optimization},
Ph.D.~thesis, MIT, 1989.

\bibitem{GK}
K.\ Goebel and W.A.\ Kirk,
\emph{Topics in Metric Fixed Point Theory},
Cambridge University Press, 1990.

\bibitem{GR}
K.\ Goebel and S.\ Reich,
\emph{Uniform Convexity, Hyperbolic Geometry, and 
Nonexpansive Mappings},
Marcel Dekker, 1984.

\bibitem{Meyer}
C.D.\ Meyer,
\emph{Matrix Analysis and Applied Linear Algebra},
SIAM, 2000.

\bibitem{Pazy} 
A.\ Pazy, 
Asymptotic behavior of contractions in Hilbert space, 
\emph{Israel Journal of Mathematics}~9, 235--240 (1971).

\bibitem{Rock70}
R.T.\ Rockafellar,
\emph{Convex Analysis},
Princeton University Press, Princeton, 1970.

\bibitem{Rock98}
R.T.\ Rockafellar and R.J-B\ Wets,
\emph{Variational Analysis},
corrected 3rd printing, Springer, 2009.
%
\bibitem{Simons1}
S.\ Simons,
\emph{Minimax and Monotonicity},
Springer, 1998.

\bibitem{Simons2}
S.\ Simons,
\emph{From Hahn-Banach to Monotonicity},
Springer, 2008.


\bibitem{Zalinescu}{C.\ Z\u{a}linescu},
\emph{Convex Analysis in General Vector Spaces},
World Scientific Publishing, 2002.

\bibitem{Zeidler2a}
E.\ Zeidler,
\emph{Nonlinear Functional Analysis and Its Applications II/A:
Linear Monotone Operators},
Springer, 1990.

\bibitem{Zeidler2b}
E.\ Zeidler,
\emph{Nonlinear Functional Analysis and Its Applications II/B:
Nonlinear Monotone Operators},
Springer, 1990.

\bibitem{Zeidler1}
E.\ Zeidler,
\emph{Nonlinear Functional Analysis and Its Applications I:
Fixed Point Theorems},
Springer, 1993.



\end{thebibliography}
\end{document}